\documentclass[11pt]{amsart}
\usepackage{amsbsy,amssymb,amsthm,amsmath}
\usepackage[pagebackref=true,colorlinks=true,linkcolor=forestgreen,citecolor=blue]{hyperref}
\usepackage{color}
\usepackage{epsfig,graphics}
\usepackage[T1]{fontenc}
\newtheorem{theorem}{Theorem}[section]
\newtheorem{lemma}[theorem]{Lemma}
\newtheorem{corollary}[theorem]{Corollary}
\newtheorem{proposition}[theorem]{Proposition}
\newtheorem{remark}[theorem]{Remark}

\makeatletter
\def\blfootnote{\xdef\@thefnmark{}\@footnotetext}
\makeatother

\newcommand{\R}{\mathbb{R}}

\newcommand{\be}{\begin{equation}}
\newcommand{\ee}{\end{equation}}

\definecolor{forestgreen}{cmyk}{0.91,0,0.88,0.12}
\definecolor{darkorange}{rgb}{1.00,0.55,0.00}
\definecolor{violet}{rgb}{0.5,0,0.8}
\usepackage{geometry}
\begin{document}

\title[Comparison principles and applications...]
{Comparison principles and applications
to mathematical modelling
of vegetal meta-communities}

\author{%
  Gauthier Delvoye, Olivier Goubet, Frédéric Paccaut
}

\maketitle
\noindent\blfootnote{\textbf{MSC 2010}: 60J60, 92D25, 35B51.}
\noindent\blfootnote{\textbf{Keywords}: Comparison principle, mathematical modelling for metacommunities, Markov chains, diffusion equations}

\begin{abstract}
This article partakes of the PEGASE project
the goal of which is a better understanding of the mechanisms explaining the behaviour
of species living in a network of forest patches linked by ecological corridors (hedges for instance).
Actually we plan to study the effect of the {\it fragmentation} of the habitat
on biodiversity.
A simple neutral model for the evolution of abundances in a vegetal
metacommunity is introduced. Migration between the communities
is explicitely modelized in a deterministic way, while the reproduction
process is dealt with using Wright-Fisher models, independently within
each community. The large population limit of the model is considered.
The hydrodynamic limit of this split-step method is proved to be the solution of a partial differential
 equation with a deterministic part coming from the migration
process and a diffusion part due to the Wright-Fisher process.  Finally, the diversity of the metacommunity is adressed
through one of its { indicators}, the mean extinction time of a species. At
the limit, using classical comparison principles, the exchange process
between the communities is proved to slow down extinction.
This shows that the existence of corridors seems to be good for the biodiversity.
\end{abstract}

\maketitle

\section{Introduction}\label{intro}

This article partakes of a research program aimed at understanding
the dynamics of a fragmented landscape composed of forest patches
connected by hedges, which are ecological corridors.
When dealing with the dynamics of a metacommunity
at a landscape scale we have to take into account
the local competition between species and the { possible}
migration of species.

We are interested here in the mathematical modelling of two species, on two forest patches
linked by some ecological corridor. We model the evolution by a {\it splitting method},
performing first the exchange process (see the definition of the corresponding Markov chain
in the sequel) on a small time step, and then we perform independently on each station
a birth/death process according to the Wright-Fisher model, and we reiterate.

Our first mathematical result is to compute the limit equation of this modelling
when the time step goes to $0$  and the size of the population diverges to $\infty$. This issue, the {\it hydrodynamic limit},
i.e. to pass from the mesoscopic scale to the macroscopic one
received increasing interest in the last decades (see for instance in various contexts
\cite{blondel}, \cite{er}, \cite{pgj}).
As our main results on extinction times { do not} require the convergence in law of the processes, instead of using a martingale problem (\cite{ethier_martingale}), we prove directly the convergence
of operators towards a diffusion semi-group (\cite{ethier_semigroup}).
We find a deterministic diffusion-convection
equation, where the drift comes from the exchange process, while the diffusion
comes from the limit of the Wright-Fisher process.
We point out here that the fact that the diffusion operator $L_d$
satisfies a non standard {\it comparison principle} (or a maximum principle) is instrumental:
{ first the comparison principle ensures the uniqueness of the limit of the approximation
process and then the definition of the Feller diffusion process. Then
this comparison principle yields our second result that} is concerned with the comparison of the extinction time
of one species for a system with exchange and a system without exchanges.
Assuming that the discrete extinction time converges, we prove that the limit is solution of the equation $-L_d \tau=1$.
Taking advantage once again of comparison principles,
we prove that the exchange process slows down the extinction time of one species.
Thus, the fragmentation of the habitat seems to be good to the biodiversity.

This article outlines as follows. In a second section
we describe the modelling at mesoscopic scale.
We couple a Wright-Fisher model for the evolution of the abundances together with an exchange process.
 The third section is devoted to the large population limit of the discrete process.
In a fourth section we discuss the issues related
to the extinction time; we compare the extinction time of one species with and without exchange process.
In a final section we draw some conclusion and prospects for ecological issues,
and we address the question of convergence in law for our model.

\section{The mathematical model}
\label{sec:1}

\subsection{Modelling the exchange between patches}
\label{subsec:1.1}

Consider two patches that have respectively the capacity to host $(N_1,N_2)$ individuals,
to be chosen into two different species $\alpha$ and $\beta$.
Set $(y_1^n,y_2^n)$ for the numbers of individuals of type $\alpha$, respectively in patch $1$ and $2$,
at time $n\delta t$, { i.e. after $n$ iterations and  $\delta t$ is the time that will be defined below.}

The exchange process is then simply modelled by
 \begin{equation}\label{echangepopulation}\begin{split}
y_1^{n+1}=(1-\kappa d \delta t)y_1^n+\kappa \delta t y_2^n, \\
y_2^{n+1}=(\kappa d \delta t)y_1^n+(1-\kappa\delta t) y_2^n,
\end{split}\end{equation}

\noindent where $\kappa$ is the instantaneous speed of exchanges
and $d =\frac{N_2}{N_1}$ represents the distortion between the patches (the ratio
between the hosting capacities); we may assume without loss of generality that $d\leq1$. With this modelling,  and assuming
that $\kappa \delta t\leq 1$, it is easy to check that

\begin{itemize}
\item The { set} $[0,N_1]\times [0,N_2]$ is { mapped into itself, i.e stable,} by the exchange process.
\item The total population of individuals of type $\alpha$, $y_1^n+y_2^n$,
is conserved.
\item If we start with only individuals of species $\alpha$ (respectively $\beta$) then
we remain with only individuals from $\alpha$ (respectively $\beta$); this
reads $(N_1,N_2)\mapsto (N_1,N_2)$ (respectively $(0,0)\mapsto (0,0)$).
\end{itemize}

Set $x=(x_1=\frac{y_1}{N_1}, x_2=\frac{y_2}{N_2}) $ belonging to $ \mathcal{D}=[0,1] ^ 2 $ for the population
densities of a species $\alpha$ on two separate patches and $x^n=(x_1^n,x_2^n)$ for these densities at time $n\delta t$. Then we have alternatively

\begin{equation}\label{echange}\begin{split}
x_1^{n+1}=(1-\kappa d \delta t)x_1^n+\kappa d\delta t x_2^n, \\
x_2^{n+1}=\kappa \delta t x_1^n+(1-\kappa\delta t) x_2^n.
\end{split}\end{equation}

\noindent This reads also  $x^{n+1}=A x^n$ where $A$ is a stochastic matrix.

Consider now the { piecewise constant} c\`adl\`ag process with jumps $X \mapsto AX$ at each time step
$\delta t$. In other words, for any continuous function $f$ defined on $\mathcal{D}=[0,1]^2$
then $P^{\rm ex}_{\delta t}(f)(x)=f(Ax)$, where $P^{\rm ex}_{\delta t}$
is the transition kernel of the exchange process.

\subsection{Wright-Fisher reproduction model}
\label{subsec:1.2}

On each patch we now describe the death/birth process
that is given by the Wright-Fisher model.
The main assumption is that the death/birth process
on one patch is {\it independent} of the other one.

Consider then the first patch that may host $N_1$ individuals.
The Markov chain is then defined by the transition matrix,
written for $z_1=\frac{j_1}{N_1} \in [0,1]$

\begin{equation}\label{trans}
\mathbb{P}(x_1^{n+1}=z_1 |  x_1^n=x_1)= \begin{pmatrix} N_1 \\ j_1\end{pmatrix} x_1^{j_1}(1-x_1)^{N_1-j_1}.
\end{equation}
\noindent Since the two Wright-Fisher processes are independent, the corresponding transition kernel reads

\begin{equation}\label{wf}
	P^{\rm wf}_{\delta t}(f)(x)=\sum_{j_1=0}^{N_1}\sum_{j_2=0}^{N_2}\begin{pmatrix} N_1 \\ j_1\end{pmatrix}\begin{pmatrix} N_2 \\ j_2\end{pmatrix}  x_1^{j_1}(1-x_1)^{N_1-j_1}x_2^{j_2}(1-x_2)^{N_2-j_2}f\left(\frac{j_1}{N_1},\frac{j_2}{N_2}\right),
\end{equation}
\noindent for any function $f$ defined on $\mathcal D=[0,1]^2$.
Notice that $P^{\rm wf}_{\delta t}$ is a two-variable version of the usual Bernstein polynomials. In the sequel, we will also use the notation $B_N(f)$ and write for the sake of conciseness
\[
\begin{pmatrix} N \\ j\end{pmatrix}  x^{j}(1-x)^{N-j}f\left(\frac{j}{N}\right)=\begin{pmatrix} N_1 \\ j_1\end{pmatrix}\begin{pmatrix} N_2 \\ j_2\end{pmatrix} x_1^{j_1}(1-x_1)^{N_1-j_1}x_2^{j_2}(1-x_2)^{ N_2 -j_2}f\left(\frac{j_1}{N_1}, \frac{j_2}{N_2}\right).
\]

\subsection{The full disrete model}
\label{subsec:1.3}
Starting from the state $x=(x_1,x_2)$, during a time step, we apply first the exchange process and then the Wright-Fisher reproduction process. In this way, the sequence of random variables $x^n$ is a Markov chain with state space $\{0,\frac{1}{N_1},\cdots,1\}\times\{0,\frac{1}{N_2},\cdots,1\}$ and the transition kernel reads as follows
\[
	\mathbb{E}(f(x^{n+1})|x^n=x)=P^{\rm wf}_{\delta t}P^{\rm ex}_{\delta t}(f)(x)=\sum_{j}\begin{pmatrix} N \\ j\end{pmatrix}  x^{j}(1-x)^{N-j}f\circ A\left(\frac{j}{N}\right)=B_N(f\circ A)(x).
\]

\section{From discrete model to continuous one}
\label{sec:3}

We consider the same scaling as for the Wright-Fisher usual
model, that is $N_1 \delta t=1$. We set $N=N_1$ in the sequel to simplify the notations.
We may consider either the c\`adl\`ag process associated to the reproduction-exchange discrete process
defined by $\overline{x}^t=x^n$ if $n\delta t \leq t < (n+1)\delta t$ or the continuous piecewise linear
function $x^t$ such that $x^t=x^n$ for $t = n\delta t$. We consider
an analogous interpolation in space in order to deal with function
that are defined on $[0,T]\times \mathcal D$ where $T>0$ is given.

We set $M=\begin{pmatrix} d & -d \\-1 & 1 \end{pmatrix}$ and then
	$A=Id -\frac \kappa N M$. For a given continuous function $f$ that vanishes
at $(0,0)$ and $(1,1)$, we now define the sequence of functions

\begin{equation}\label{espcon}
u_N(t,x)= \mathbb{E}( f(x^t) | x^0=x).
\end{equation}

\noindent We may also use analogously $\overline{u}_N(t,x)= \mathbb{E}( f(\overline{x}^t) | x^0=x)
=(P_{\delta t}^{\rm wf}P_{\delta t}^{\rm ex})^n(f)(x) .$
{ The fonctions $u_N$ and $\overline{u}_N$ represent
the average densities of the species at a macroscopic level.
If $X^N$ is the Lagrangian representation of the densities,
then $u_N$ represents the densities in Eulerian variables. }


\subsection{Statement of the result}
\label{subsec:2.1}

\begin{theorem}\label{limit} Let T>0 be fixed.
Assume $f$ is a function of class $C^2$ on $\mathcal D$, that vanishes at $(0,0)$ and $(1,1)$.
The sequence $u_N$ converges uniformly in $[0,T]\times \mathcal{D}$ to the
unique solution $u$
of the diffusion equation
$$\partial_t u=L_d u,$$
\noindent where $L_d$ is defined as, for $x=(x_1,x_2)$,
$$L_d u (x)= \frac{x_1(1-x_1)}{2}u_{x_1x_1}(x)+\frac{x_2(1-x_2)}{2d}u_{x_2x_2}(x)-\kappa Mx.\nabla u(x),$$
\noindent and with initial data $u(0,x)=f(x)$.
\end{theorem}

\begin{remark}
	We may have proved that the c\`adl\`ag process associated to the reproduction-exchange process $\overline{u}_N$
converges to a diffusion equation. We will discuss this in the sequel.
Besides, we prove the convergence results for a sufficiently smooth $f$, and we will extend
in the sequel the definition of a mild solution to the equation for functions
$f$ in the Banach space $E=\{ f \in C(\mathcal D) ; f(0,0)=f(1,1)=0\}.$
{ The theory for Markov diffusion process and the related PDE equations
is well developed in the litterature (see \cite{bgl}, \cite{ethier_kurtz}, \cite{meleard}
and the references therein). The particularity of our diffusion equation is that the boundary
of the domain is only two points.}
\end{remark}

\subsection{Proof of Theorem \ref{limit}}\label{durelimite}
\label{subsec:2.2}
The proof of the theorem is divided into several lemmata.
{ The first lemma describes in a way how the discrete process is close to a martingale.
}

\begin{lemma}\label{martin}
The conditional expectation of the discrete reproduction-exchange process is
\begin{equation}\label{pisco}
\mathbb{E}(x^{n+1} | x^n)=Ax^n.
\end{equation}
\noindent As a consequence $\mathbb{E}(x^{n+1}-x^n | x^n)= o(1)$ when $N$ diverges to $\infty$.
\end{lemma}

\begin{proof}

Using the properties of the Bernstein polynomials,
	$$  \mathbb{E}(x^{n+1} | x^n)=\sum_{j}\begin{pmatrix} N \\ j\end{pmatrix} (x^n)^{j}(1-x^n)^{N-j}A\begin{pmatrix} \frac{ j_1}{ N_1} \\ \frac{ j_2}{ N_2}\end{pmatrix}
		=Ax^n.$$

\noindent Then the proof of the lemma is completed,
observing that $A-Id=o(1)$.

\end{proof}


The following lemma is useful to prove that $x^t$ and $\overline{x}^t$
are close.

\begin{lemma}\label{mom2}
There exists a constant $C$ such that
$$\mathbb{E}(| x^{n+1}-x^n|^2)\leq C N^{-1}.$$
\end{lemma}

\begin{proof}

\noindent Since $|A\frac jN|^2=|\frac jN|^2(1+O(||A-Id||)$, then the following conditional expectation reads
 $$\mathbb{E}(| x^{n+1}|^2 |x^n)=\sum_{j}\begin{pmatrix} N \\ j\end{pmatrix} x^{j}(1-x)^{N-j}|A\frac jN|^2=
|x^n|^2+O(||A-Id||).$$

We expand the $\ell^2$ norm in $\R^2$ as
$$ | x^{n+1}-x^n|^2=| x^{n+1}|^2-2(x^n,x^{n+1})+|x^n|^2. $$
We first have by linearity and by the Lemma \ref{martin} above that

$$\mathbb{E}((x^{n+1},x^{n})  | x^n)=(Ax^n,x^n).$$

\noindent Therefore
\begin{equation}\label{chile2}
\mathbb{E}(| x^{n+1}-x^n|^2| x_n)=2(x^n,x^n-Ax^n)+O(||I-A||)=O(||I-A||)
\end{equation}

\noindent that completes the proof of the lemma.

\end{proof}

\noindent The next statement is a consequence of the inequality
$|\overline{x}^t-x^t|\leq |x^n-x^{n+1}|$ for $t\in (n\delta t, (n+1)\delta t)$
and of the previous lemma

\begin{corollary}\label{pdc}
The processes $x^t$ and $\overline{x}^t$ are asymptotically close, i.e.
there exists a constant $C$ such that
$$ \mathbb{E}(|\overline{x}^t-x^t|^2)\leq C N^{-1}.$$
\end{corollary}

\noindent As a consequence, when looking for the limit when $N$
diverges towards $+\infty$ of the process, we may either
work with $x^t$ or $\overline{x}^t$.

The next lemma is a compactness result on the bounded sequence
$u_N$ defined in \eqref{espcon}.

\begin{lemma}\label{ascoli}
There exists a constant $C$ that depends on $||f||_{lip}$ and on $T$ such that
for any, $x,y$ in $\mathcal D$ and $s,t$ in $[0,T]$,
$$ |u_N(t,x)-u_N(t,y)| \leq C |x-y|, $$
$$ |u_N(t,x)-u_N(s,x)|\leq C |t-s|^\frac12 . $$
\end{lemma}

{
\begin{remark}
Since the constants $C$ do not depend on $N$
we can infer letting $N\rightarrow \infty$
some extra regularity results for $u$, assuming that $f$ is Lipschitz.
\end{remark}
}

\begin{proof}
We begin with the first estimate. Introduce $n$ such that $n \delta t \leq t <(n+1)\delta t$.
Set $y^t$ for the process that starts from $y=y^0$.
$$ |x^t-y^t|\leq \max( |x^n-y^n|,  |x^{n+1}-y^{n+1}|), $$
\noindent therefore, proving the first inequality for
$\overline{u}_N$ (which amounts to controlling $|x^n-y^n|)$) will imply the inequality for $u_N$.
Due to the properties
of Bernstein's polynomials we have that

\begin{equation}\label{home1}
|\partial_{x_1}P^{\rm wf}_{\delta t}(f)(Ax)|\leq N ||A|| \omega(f,\frac{1}{N}),
\end{equation}

\noindent where $\omega(f,\frac{1}{N})$ is the modulus of continuity of $f$.
Then, using that $||A-Id||\leq C N^{-1}$, we infer that

\begin{equation}\label{home2}
|\partial_{x_1}P^{\rm wf}_{\delta t}(f)(Ax)|\leq ||f||_{lip} (1+\frac{C}{N}).
\end{equation}

\noindent Iterating in time we have that,

\begin{equation}\label{home3}
|\partial_{x_1}(P^{\rm wf }_{\delta t }P^{\rm ex}_{\delta t})^m(f)(x)|\leq ||f||_{lip} (1+\frac{C}{N})^m\leq \exp(CT) ||f||_{lip}.
\end{equation}

\noindent The other derivative is similar and then we infer from this computation
that the first inequality in the statement of Lemma \ref{ascoli} is proved.

We now proceed to the proof of the second one.
Introduce the integers $m,n$ such that
$m\delta t\leq s < (m+1)\delta t$ and $n\delta t\leq t < (n+1)\delta t$.
Using that
$$ |x^t-x^s|^2\leq 9 (|x^{m+1}-x^s|^2+|x^{m+1}-x^n|^2+|x^n-x^t|^2),$$
\noindent and that $|x^n-x^t|\leq (t-\frac {n}{\delta t}) |x^n-x^{n+1}|$
we just have to prove the inequality for $\frac {t}{ \delta t}$ and $\frac {s} {\delta t}$ in $\mathbb{N}$.
Introduce the increment $y^j=x^{j+1}-x^j$.
We have that, for $m\leq i,j\leq n$

\begin{equation}\label{rhum1}
\mathbb{E}(|x^n-x^m|^2)=\sum_j \mathbb{E}( |y^j|^2)+2\sum_{i<j} \mathbb{E}(y^i,y^j).
\end{equation}

\noindent On the one hand, by Lemma \ref{mom2} we have that the first term
in the right hand side of \eqref{rhum1} is bounded by above by $\frac{C(m-n)}{N}$.
On the other hand, using that the $\mathbb{E}(y^j | x^j)=(A-Id)x^j$
then

\begin{equation}\label{rhum2}
\sum_{i<j} \mathbb{E}(y^i,y^j)=\sum_{i<j}\mathbb{E}(y^i,(A-Id)x^j)=\sum_j\mathbb{E}(x^{j}-x^{m},(Id-A)x^j).
\end{equation}

\noindent Since $||Id-A||\leq C N^{-1}$ then the right hand side of \eqref{rhum2}
is also bounded by above by $\frac{C(m-n)}{N}$.
This completes the proof of the lemma.

\end{proof}

\noindent Thanks to Ascoli's theorem, up to a subsequence
extraction, $u_N$ converges uniformly to a continuous function
$u(t,x)$. We now prove that $u$ is solution
of a diffusion equation whose infinitesimal generator
is defined as the limit of $N(P^{\rm wf}_{\delta t}P^{\rm ex}_{\delta t}-Id)$.

\begin{lemma}\label{infgen}
Consider $f$ a function of class $C^2$ on $\mathcal D$ that vanish at $(0,0)$ and $(1,1)$.
Then
$$ \lim_{\delta t \rightarrow 0^+} N(P^{\rm wf}_{\delta t}P^{\rm ex}_{\delta t}(f(x))-f(x))=\lim_{N\rightarrow \infty} N(B_N(f\circ A(x)-f(x))=L_d f(x),$$
\noindent where $L_d $ is defined in Theorem \ref{limit}.
\end{lemma}

\begin{proof}
Due to Taylor formula

\begin{equation}\label{taylor}
P^{\rm ex}_{\delta t}(f)(x)=f(x)-\kappa \delta t (Mx.\nabla f)(x) +O((\delta t)^2),
\end{equation}

\noindent where $(O(\delta t)^2)$ is valid uniformly in $x$ in $\mathcal D$.

\noindent Using that the linear operator $P^{\rm wf}_{\delta t}$ is positive and bounded by $1$ we then have

\begin{equation}\label{taylor2}
P^{\rm wf}_{\delta t}P^{\rm ex}_{\delta t}(f)(x)=(P^{\rm wf}_{\delta t}f)(x)-\kappa \delta t P^{\rm wf}_{\delta t}(Mx.\nabla f)(x) +O((\delta t)^2),
\end{equation}

{
\noindent The well-known properties of Bernstein polynomials (see \cite{bus}) entail
that uniformly in $x$
$$ [P^{\rm wf}_{\delta t}(Mx.\nabla f)(x)-Mx.\nabla f(x)|\leq C \sqrt{\delta t}. $$
On the other hand, the operator $P^{\rm wf}_{\delta t}$
is the tensor product of two one-dimensional Bernstein operators.
Then by Voronovskaya-type theorem (see \cite{bus}), for
$f(x)=f_1(x_1)f_2(x_2)$ we have the uniform convergence of $N( (P^{\rm wf}_{\delta t}f)(x)-f(x))$
to $\frac{x_1(1-x_1)}{2}f_{x_1x_1}+\frac{x_2(1-x_2)}{2d}f_{x_2x_2}$. By density
of the linear combinations of tensor products $f_1(x_1)f_2(x_2)$ this result
extend to general $f$ as
\begin{equation}\label{taylor3}
\lim_{\delta t \rightarrow 0^+}\frac{P^{\rm wf}_{\delta t}(f)(x)-f(x)}{\delta t}=\frac{x_1(1-x_1)}{2}f_{x_1x_1}+\frac{x_2(1-x_2)}{2d}f_{x_2x_2}.
\end{equation}
}

\noindent Denoting $\Delta_d$ the diffusion operator defined by the right hand side of \eqref{taylor3},
the Kolmogorov limit equation of our coupled Markov process is

\begin{equation}\label{taylor4}
\partial_t u -\Delta_d u=-\kappa (Mx.\nabla u)(x),
\end{equation}

\noindent with initial data $u(0,x)=f(x)$. Let us observe that
$u$, the limit of $\mathbb{E}(f(x^t) | x^0=x)$, vanishes at two points $(0,0)$ and $(1,1)$ in the boundary $\partial \mathcal D$.

\end{proof}

We now complete the proof of the Theorem. Considering
$f$ such that the convergence in Lemma \ref{infgen} holds. Then, for $n\leq t N<n+1$,


\begin{equation}\label{liguria2}
\overline{u}_N(t,x)=f(x)+\sum_{k=0}^{n-1}\int_{k\delta t}^{(k+1)\delta t}( N (P^{\rm wf}_{\delta t}P^{\rm ex}_{\delta t}-Id)(\overline{u}_N(s,.)))(x)ds.
\end{equation}

\noindent Using the uniform convergence of $\overline{u}_N$, Lemma \ref{infgen} and a recurrence
on $n$ we may prove that at the limit

\begin{equation}\label{semigroup}
u(t)=f+\int_0^t L_d u(s)ds,
\end{equation}

\noindent where we have omitted the variable $x$ for the sake of convenience.

We now state a result that ensures the uniqueness of
a solution to the diffusion equation \eqref{semigroup}.
Such a solution is a solution to the diffusion equation in a weak PDE sense.

{ Introduce $D(L_d)=\{ f \in E; \frac{P^{\rm wf}_{\delta t}P^{\rm ex}_{\delta t}(f)(x)-f(x)}{\delta t}\rightarrow L_d f \; {\rm in} \; E\}$.

\begin{remark}
We precise here the regularity of the functions $f$ in $D(L_d)$.
Since $L_d$ is a strictly elliptic operator on any compact
subset of the interior of $\mathcal{D}$ then
$f$ is $C^2(\mathring{\mathcal{D}}) \cap C(\mathcal{D})$
(see \cite{gt}). The regularity of $f$ up to the boundary
is a more delicate issue (see \cite{lsu}, \cite{lieberman}). Besides, to determine exactly what is the domain of $L_d$ is a difficult issue.
For PDEs the unbounded operator is also determinated by its
boundary conditions. Here we have boundary conditions of Ventsel'-Vishik
type, that are integro-differential equations on each side of the square
linking the trace of the function $f$ and its normal derivative. This is beyond the
scope of this article.
\end{remark}

}

\begin{theorem}[Comparison Principle]\label{pmax}
\begin{itemize}
\item Parabolic version: Consider a function $u$ in $C(\R^+, D(L_d))$ that satisfies
\begin{itemize}
\item $u_t-L_du \geq 0$ in $\R^+\times [0,1]^2$,
\item $u(0,x)=f(x) \geq 0$ for $x$ in $[0,1]^2$,
\end{itemize}
\noindent then $u(t,x)\geq 0$.
\item { Elliptic version:
Consider $u(x)$ in $D(L_d)$ that satisfies $-L_du \geq 0$ in $[0,1]^2$.
Then $u(x)\geq 0$.  }
\end{itemize}
\end{theorem}

\noindent We postpone the proof of this theorem until the end of this section.
We point out that a comparison principle for $L_d$
is not standard since it requires only information on two points $\{(0,0), (1,1)\}$
in $\partial\mathcal D$ and not on the whole boundary.

Theorem \ref{pmax} implies uniqueness of the limit solution. Therefore the whole sequence $u_N$ converge and the semigroup
is well defined. Actually, setting $S(t)f=u(t)$ we then have defined for smooth $f$
the solution to a Feller semigroup (see \cite{bgl}) as follows
\begin{enumerate}
\item $S(0)=Id$.
\item $S(t+s)=S(t)S(s)$.
\item $||S(t)f-f||_E\rightarrow 0$ when $t\rightarrow 0^+$
\item  $||S(t)f||_E\leq ||f||_E. $
\end{enumerate}

\noindent The second property comes from uniqueness,
the last one passing to the limit in
$$||(P^{\rm wf}_{\delta t}P^{\rm ex}_{\delta t})^nf||_{L^\infty}\leq ||f||_{L^\infty}.$$
\noindent The third one is then simple. The third
property allows us to extend the definition
of $S(t)$ to functions in $E$ by a classical density argument.
Then we have a Feller semigroup in $E$ that satisfies
the assumptions of the Hille-Yosida theorem
(see \cite{brezis}).

\subsection{Proof of the comparison principle}\label{maisonmarie}

{ We begin with the comparison principle
for the parabolic operator. We use that
$C^2(\mathcal{D})$ is dense in $D(L_d)$, i.e. that any function
$u$ in $D(L_d)$ can be approximated in $E$ by smooth functions $u_k$ up to the boundary,
and such that $Lu_k$ converges uniformly on any compact subset of $\mathring{\mathcal D}$.
We then prove the comparison principle for smooth functions
and we conclude by density. }

Consider $u$ as in the statement of the Theorem for a $C^2$ initial data $f$.
Consider $\varepsilon$ small enough.
Set $\mathcal{P}=\partial_t-L_d$.
Set $\psi(x)=(x_1+dx_2)(d+1-x_1-dx_2)$ and $\theta(x)=(x_1-x_2)^2$.
Introduce the auxiliary function

\begin{equation}\label{auxiliary}
v(t,x)=u(t,x)+\varepsilon \psi(x)+\varepsilon^2 \theta(x)+\varepsilon^3.
\end{equation}

\noindent We prove below that $v(t,x)\geq 0$ for all $t$ and $x$.
Since $u$ belongs to $D(L_d)$ then $v(t,x)=\varepsilon^3$ at the corners $x\in \{(0,0), (1,1)\}$.
We also have $v(0,x)\geq \varepsilon^3$.

Let us then argue by contradiction.
Introduce ${t_0}=\inf \{ t>0; \exists x \in[0,1] ; \; v(t,x)<0\}.$ 
Then there exists $x_0$ such that $v({t_0},x_0)=0$. Notice that $x_0\notin\{(0,0),(1,1)\}$.
We shall discuss below different cases according to the location of $x_0$.

\noindent First case: $x_0$ belongs to the interior of $\mathcal D$.

We then have $v_t(t_0,x_0)\leq 0$, $v_{x_1}(t_0,x_0)=v_{x_2}(t_0,x_0)=0$,
and $v_{x_1x_1}(t_0,x_0), v_{x_2x_2}(t_0,x_0)\geq 0$.
Therefore $\mathcal{P}v({t_0},x_0)\leq 0$

\begin{equation}\label{cors1}
0\leq \mathcal{P}u({t_0},x_0)\leq \mathcal{P}v({t_0},x_0)+\varepsilon L_d(\psi+\varepsilon\theta)(x_0).
\end{equation}

\noindent Let us observe that if $\varepsilon$ is chosen small enough

\begin{equation}\label{cors2}\begin{split}
L_d(\psi+\varepsilon\theta)(x_0)=-\left( x_1(1-x_1)+dx_2(1-x_2)\right)+\\
\varepsilon \left( x_1(1-x_1)+\frac{1}{d}x_2(1-x_2)\right)-\kappa \varepsilon (d+1) \theta(x)^2<0.
\end{split}\end{equation}

\noindent Second case: $x_0$ belongs to $\partial \mathcal D$ but the four corners.

We may assume that $x_0=(0,x_2)$ the other cases being similar.
We have that $v_t(t_0,x_0)\leq 0$, $v_{x_2}(t_0,x_0)=0$,  $v_{x_1}(t_0,x_0)\geq 0$
and $v_{x_2x_2}(t_0,x_0)\geq 0$.
Therefore $\mathcal{P}v({t_0},x_0)\leq 0$.

We then have as in \eqref{cors1} that $0\leq L_d(\psi+\varepsilon\theta)(x_0)$.
Computing $L_d(\psi+\varepsilon\theta)(x_0)=-\varepsilon\kappa(d+1)\theta^2(x_0)<0$
gives the contradiction.

\noindent Third case: $x_0=(0,1)$  (the case $(1,0)$ is similar).

We have that $v_t(t_0,x_0)\leq 0$, $v_{x_2}(t_0,x_0)\leq 0\leq v_{x_1}(t_0,x_0)$.
Therefore $\mathcal{P}v({t_0},x_0)\leq 0$.
We then have as in \eqref{cors1} that $0\leq L_d(\psi+\varepsilon\theta)(x_0)$.
Computing $L_d(\psi+\varepsilon\theta)(x_0)=-\varepsilon\kappa(d+1)\theta^2(x_0)<0$
gives the contradiction.

We now conclude. since $v$ is nonnegative we have

\begin{equation}\label{cors4}
\inf_{[0,+\infty)\times \mathcal D} u \geq -\varepsilon ||\psi+\varepsilon \theta||_{L^\infty}-\varepsilon^3.
\end{equation}

\noindent Letting $\varepsilon$ goes to $0$ completes the proof.

{ Let us prove the elliptic counterpart of the result
for a smooth function $u$ (we also proceed by density). Set as above $v(x)=u(x)+\varepsilon \psi(x)+\varepsilon^2 \theta(x)$.
Introduce $x_0$ where $v$ achieves its minimum, i.e. $v(x_0)=\min_{\mathcal{D}}v(x).$.
First if $x_0$ belongs to the interior of $\mathcal{D}$, then $L_dv(x_0)>0$ and we have a contradiction.
We disprove the case where $x_0$ belongs to the boundary but $\{(0,0), (1,1)\}$
exactly as in the evolution equation case. Assume first that
$x_0$ belongs to $\partial \mathcal D$ but the four corners; for instance $x_0=(0,x_2)$.
We have that $v_{x_2}(x_0)=0$,  $v_{x_1}(x_0)\geq 0$
and $v_{x_2x_2}(x_0)\geq 0$. Therefore $L_d v(x_0)\geq 0$.
Then $L_d(\psi+\varepsilon\theta)(x_0)=<0$ gives the contradiction.
Assume then that  $x_0=(0,1)$.
We have that $v_{x_2}(x_0)\leq 0\leq v_{x_1}(x_0)$.
Therefore $-L_dv(x_0)\leq 0$.
We then have that $0\leq L_d(\psi+\varepsilon\theta)(x_0)$.
Computing $L_d(\psi+\varepsilon\theta)(x_0)=<0$
gives the contradiction.

\begin{corollary}\label{pmp}
Actually $L_d$ satisfies the {\it positive maximum principle} (PMP).
If $u$ in $D(L_d)$ achieves its minimum in $x_0$ in the interior
of $\mathcal{D}$ then $L_du(x_0)\geq 0$. This is standard for infinitesimal generator of Feller semigroups
(see \cite{bony}).
\end{corollary}

}

\section{Extinction time}\label{ET}

\subsection{Hydrodynamic limit of the extinction time}

We handle here the convergence of the discrete extinction time towards
the solution of an elliptic equation. To begin with, recall that the discrete process describing the evolution of the densities of population (migration and reproduction at each time step) is a Markov chain with state space $\{0,\frac{1}{N},\dots,1\}\times\{0,\frac{1}{N},\dots,1\}$ for which $(0,0)$ and $(1,1)$ are absorbing states.
These two absorbing states correspond to the extinction of a species.
Let us introduce the hitting time $\Theta_N$ that is the random time
when the Markov chain reaches the absorbing states, i.e. the \emph{extinction time}.
Since the restriction of the chain to the non absorbing states is irreducible and
since there is at least one positive transition probability from the non absorbing states to the absorbing states
then this hitting time is almost surely finite. { This result is standard
for Markov chains with finite state space (see \cite{bremaud}, \cite{chafai} and the references therein).  }

Let $U$ be the complement of the trapping states $(0,0)$ and $(1,1)$.
Consider the vector $T_N$ defined as the conditional expectation $(T_N)_{\frac{j}{N}\in U}=\mathbb{E}_{\frac j N}(\Theta_N)$
of this hitting time and denote by $\tilde P_N$ or $\tilde P$ the restriction of the transition matrix to $U$.

 Then for $x\in U$, denoting $\mathbb{P}_x$ the conditional probability, we have using Markov property
 and time translation invariance
\begin{eqnarray*}
\mathbb E_x(\Theta_N)&=&\sum_{k=1}^\infty\frac{k}{N}\mathbb P_x(\Theta_N=\frac{k}{N})=\frac1N\mathbb P_x(\Theta_N=\frac1N)+\sum_{k=2}^\infty\frac{k}{N}\mathbb P_x(\Theta_N=\frac{k}{N})\\
&=&\frac1N\mathbb P_x(\Theta_N=\frac1N)+\sum_{k=2}^\infty\frac{k}{N}\sum_{y\in U}\mathbb P_x(\Theta_N=\frac{k}{N},x^1=y)\\
&=&\frac1N\mathbb P_x(\Theta_N=\frac1N)+\sum_{k=2}^\infty\frac{k}{N}\sum_{y\in U}\mathbb P(\Theta_N=\frac{k}{N}|x^1=y)\mathbb P_x(x^1=y)\\
&=&\frac1N\mathbb P_x(\Theta_N=\frac1N)+\sum_{y\in U}\tilde P_{x,y}\sum_{k=2}^\infty(\frac{k-1}{N}+\frac1N)\mathbb P_y(\Theta_N=\frac{k-1}{N})\\
&=&\left( \frac1N\mathbb P_x(\Theta_N=\frac1N)+\frac1N \sum_{y\in U} \tilde P_{x,y}\right) +\sum_{y\in U}\tilde P_{x,y}\mathbb E_{y}(\Theta_N)\\
&=&\frac1N+\sum_{y\in U}  \tilde P_{x,y}\mathbb E_{y}(\Theta_N).
\end{eqnarray*}
\noindent This is equivalent to
\begin{equation}\label{jpo1}
N(Id-\tilde P_N)T_N=\begin{pmatrix} 1\\...\\1\end{pmatrix}.
\end{equation}

\noindent We are now interested in the limit of $T_N$ when $N$ diverges towards $\infty$.

{Let us recall that for the one dimensional Wright-Fisher
process the expectation of the hitting time starting from $x$ converges towards
the entropy $H(x)$ (see \cite{meleard}) defined by
\begin{equation}\label{entropy}
H(x)= -2\left(x\ln x +(1-x)\ln(1-x) \right).
\end{equation}
The entropy is a solution to the equation $-\frac{x(1-x)}{2}H_{xx}=1$
that vanishes at the boundary. The proof, that can be found in Section 10 of \cite{ethier_kurtz}, uses probability tools
like the convergence in distribution of the processes and
the associated stochastic differential equation.
We believe that the same kind of tools would give the convergence of $\tau_N$ in dimension two but this is beyond the scope of this article.
Besides, for the sake of completeness we provide a proof for the convergence in distribution of our processes in Section
\ref{proba} below.

Set now $\tau_N$ for the polynomial of degree $N$
in $x_1$ and $x_2$ that interpolates $T_N$ at the points of the grid.
We have

\begin{theorem}[Extinction time]\label{et}
When $N$ diverges to $+\infty$ the sequence $\tau_N$
converges towards $\tau$ that is solution to
the elliptic equation $-L_d \tau=1$.
\end{theorem}

Assuming the convergence of $\tau_N$, the proof of the theorem is straightforward by passing to the limit in \eqref{jpo1} using Lemma \ref{infgen}.

{
\begin{remark}
We expect the function $\tau$ to be smooth up to the boundary but at the two points
$(0,0)$ and $(1,1)$. We admit here this result. This allows us to use the previous comparison
result.
\end{remark}
}

The solution of this elliptic equation in $E$, i.e.
that vanishes at $\{(0,0), (1,1)\}$ is unique due to comparison
principle (see Theorem \ref{pmax} above).
}

\subsection{Exchanges slow down extinction}

Consider now a single patch whose hosting capacity is
$N_1+N_2=(d+1)N$ for $N=\frac{1}{\delta t}$. The limit equation for the classical
Wright-Fisher related process is

\begin{equation}\label{single}
\partial_t u= \frac{z(1-z)}{2(1+d)}\partial^2_{z}u.
\end{equation}

Then the corresponding extinction time
for the Wright-Fisher process without exchange is
$\underline \tau=(d+1)H(z)$, where $z=\frac{x_1+dx_2}{1+d}$ is the corresponding
averaged starting density (see \cite{meleard}) and where $H$ is the entropy defined above \eqref{entropy}.
We shall prove in the sequel

\begin{theorem}\label{compex}
The extinction time $\underline \tau$ is a subsolution
to the equation $-L_d \tau=1$. Besides, the operator $L_d$
satisfies the comparison principle and then $\underline \tau \leq \tau$.
\end{theorem}

\begin{proof} We point out that to check that $-L_d$ satisfies
the comparison result is not obvious (see Theorem \ref{pmax}).
We first observe that the entropy \eqref{entropy}
vanishes at the boundary points $\{(0,0), (1,1)\}$.
	Setting $\underline{\tau}(x_1,x_2)=g(z)$, we have

\begin{equation}\label{camion1}
(Mx.\nabla g)(z)= (x_1-x_2) g'(z) \left( d\partial_{x_1} z- \partial_{x_2}z \right)=0,
\end{equation}

\noindent We then have

\begin{equation}\label{camion4}
-L_d \underline \tau = \frac{x_1(1-x_1)+d x_2(1-x_2)}{(1+d)z(1-z)}.
\end{equation}

\noindent Observing that by a mere computation

\begin{equation}\label{pouetpouet}
 \frac{x_1(1-x_1)+d x_2(1-x_2)}{(1+d)z(1-z)}= 1 - \frac{d(x_1-x_2)^2}{(1+d)^2z(1-z)},
\end{equation}

\noindent we have that $\underline \tau$ is a subsolution to the equation. \end{proof}

\subsection{More comparison results}

We address here the issue of the convergence of the limit extinction time
$\tau=\tau_{d,\kappa}$ defined in Section \ref{ET} when $\kappa$ or $d$ converges towards $0$.
This extinction time depends on the starting point $x$.

\begin{proposition}\label{simpson}
Assume $d$ be fixed. When $\kappa$ converges to $0$
then $\lim \tau_{d,\kappa}(x)=+\infty$ everywhere but in $x=(0,0)$ or $x=(1,1)$.
\end{proposition}

\begin{proof}
Consider here the function $V=\frac{ x_1(1-x_2)+x_2(1-x_1)}{12 \kappa}$.
This function vanishes at $x=(0,0)$ and $x=(1,1)$
and satisfies

\begin{equation}\label{sub1}
-L_d V= \frac{(x_1-x_2)( d(1-2x_2)-(1-2x_1))}{12} \leq 1.
\end{equation}

\noindent Then $V$ is a subsolution to the equation $-L_d\tau=1$
and by the comparison principle $V \leq \tau_{d,\kappa}$ everywhere.
Letting $\kappa \rightarrow 0$ completes the proof of the Proposition.
\end{proof}

\begin{proposition}\label{pfut}
Assume $\kappa$ be fixed. Then
$$\lim_{d\rightarrow 0} \tau= H(x_1)=-2x_1\ln x_1- 2 (1-x_1)\ln (1-x_1),$$
that is the extinction time for one patch.
\end{proposition}

\begin{proof}
We begin with

\begin{equation}\label{camion5}
-L_d(\tau- \underline \tau) =\frac{d(x_1-x_2)^2}{(1+d)^2z(1-z)}.
\end{equation}

Let us observe that due to \eqref{pouetpouet}
\begin{equation}\label{super11}
\frac{d(x_1-x_2)^2}{(1+d)^2z(1-z)} \leq 1.
\end{equation}

\noindent The strategy is to seek a supersolution $X$ to the equation
$-L_d \tilde X =\frac{1}{d}$ that is bounded when $d$ converges to $0$.
We first have, using the entropy function $H_2(x_1,x_2)=H(x_2)$

\begin{equation}\label{super12}
-L_d H_2=\frac1d+2\kappa(x_1-x_2)\ln\frac{x_2}{1-x_2}\geq \frac1d+2\kappa(x_1\ln x_2+(1-x_1)\ln(1-x_2)).
\end{equation}

\noindent Setting $D(x_1,x_2)=x_1x_2^d+(1-x_1)(1-x_2)^d$, we have

\begin{equation}\label{super13}
	-L_d D= \frac{1-d}{2}(x_2^{d-1}x_1(1-x_2)+(1-x_2)^{d-1}x_2(1-x_1))-\kappa d(x_1-x_2)^2(x_2^{d-1}+(1-x_2)^{d-1}).
\end{equation}

	\noindent Therefore, since we have
$$ \frac{1-d}{2}x_2^{d-1}x_1(1-x_2)-\kappa d(x_1^2-2x_1x_2+x_2^2) x_2^{d-1}\geq x_2^{d-1}x_1\left(\frac{1-d}{2}-\kappa d \right)-\frac{1-d}{2}-\kappa d, $$

\noindent we obtain, for $d$ small enough to have $(1+2\kappa )d<1$,

	$$-L_d D\geq -1 -2d\kappa + \frac{1-(1+2\kappa) d}{2}\left(x_1x_2^{d-1}+(1-x_1)(1-x_2)^{d-1} \right).$$

Gathering this inequality with \eqref{super12} and chosing
$d$ small enough such that $\frac{1-(1+2\kappa) d}{2} \geq \frac14$ holds true,
we then have

\begin{equation}\label{patatedouce}
-L_d( H_2+D)\geq  (\frac1d -1 -2d\kappa ) +x_1 ( 2\kappa \ln x_2+\frac{ x_2^{d-1}}{4} )
+ (1-x_1) (2\kappa \ln(1-x_2)+ \frac { (1-x_2)^{d-1}}{4}).
\end{equation}

\noindent Using the estimate
	
$$ \frac{1}{4x_2^{1-d}}+2\kappa \ln x_2\geq \frac{2\kappa}{1-d}(1+\ln(\frac{8\kappa}{1-d}). $$

\noindent we have that if $d$ is small enough depending on $\kappa$
then $-L_d( H_2+D)\geq \frac{1}{2d}$. Using the comparison principle
we then have that

\begin{equation}\label{super14}
0\leq \tau- \underline \tau \leq 2d(H_2+D),
\end{equation}

\noindent and we conclude by letting $d$ converge to $0$ since
$\underline \tau$ converges towards $H(x_1)$.
\end{proof}

\section{Miscellaneous results and comments}

\subsection{Discussion and prospects for ecological issues}\label{discussion}

To begin with, we have introduced a split-step model
that balances between the local reproduction of species
and the exchange process between patches. This split-step
model at a mesoscopic scale converges towards
a diffusion model whose drifts terms come from
the exchanges. This has been also observed for
instance in \cite{wakeley}.

Here we deal with a neutral metacommunity model
with no exchange with an external pool.
Hence the dynamics converge to a fixation on a single species
for large times. The average time to extinction of species is therefore an indicator of biodiversity.
Here for our simple neutral model, Theorem \ref{compex} provides a strong reckon that the exchange process is good for the biodiversity.
In some sense, the presence of two patches allows each species to establish itself
during a larger time lapse.

In a forthcoming work we plan to numerically
study a similar model but with more than two patches and several species.
We plan also to calibrate this model with data
measured in the south part of Hauts-de-France.
The main interest is to assess the role of ecological corridors to maintain biodiversity in an area.
 The question of the benefit of maintaining hedges arises when the agricultural world works for their removal to enlarge the cultivable plots.
This is one of the issue addresses by the Green and Blue Frame in Hauts-de-France.

\subsection{Convergence un distribution}\label{proba}

We address here the convergence in law/in distribution
of the infinite dimensional processes related to the $x_N^t$.
This is related to the convergence of the process towards
the solution of a stochastic differential equations; we will not develop this here.
Following \cite{stroock} or \cite{kallenberg},
it is sufficient to check the tightness of the process and
the convergence of the finite m-dimensional law.

Dealing with $\overline x_N^t$ instead of $x_N^t$,
the second point is easy. Indeed, Theorem \ref{limit} implies
 the convergence of the m-dimensional law for $m=1$. We can extend the result for arbitrary $m$ by induction using the Markov property.
For the tightness, we use the so-called Kolmogorov criterion
that is valid for continuous in time processes
(see \cite{stroock} chapter 2 and \cite{kallenberg} chapter 14);
this criterion reads in our case

\begin{equation}\label{kolmogorov}
\mathbb{E}( |x^s_N-  x^t_N|^4)\leq C |t-s|^2.
\end{equation}

This is a consequence of the following discrete estimate,
since $x^t_N$ is piecewise linear with respect to $t$,

\begin{proposition}\label{k}
There exists a constant $C$ such that for any $m<n$
\begin{equation}\label{keq}
\mathbb{E}( |x^n-x^m|^4)\leq C \frac{|n-m|^2}{N^2}.
\end{equation}
\end{proposition}

\begin{proof}

\noindent First step: using that $x^n$ is close to a true martingale.

Let us set $A=Id-\frac{\kappa}{N}M=Id-B$.
Introduce $z^0=x^0$ and $z^n=x^n+B\sum_{k<n} x^k$.
Then since $\mathbb{E}(x^{n+1} |x^n)=x^n-Bx^n,$
we have that $z^n$ is a martingale.
Moreover we have the estimate, for $0\leq m<n$

\begin{equation}\label{point}
|(z^n-x^n)-(z^m-x^m)|\leq (n-m)||B||\leq C \frac{n-m}{N}.
\end{equation}

\noindent Second step: computing the fourth moment.

To begin with we observe that, due to \eqref{point}

\begin{equation}\label{k1}
|x^n-x^m|^4\leq 4\left(|z^n-z^m|^4+C (\frac{n-m}{N})^4  \right).
\end{equation}

\noindent Therefore we just have to prove that \eqref{keq} is valid
with $z^n$ replacing $x^n$. We introduce the increment $y^j=z^{j+1}-z^j$.
We then expand as follows, setting $|.|$ and $(.,.)$ respectively
for the euclidian norm and the scalar product in $\mathbb{R}^2$.

\begin{equation}\label{k2}
\mathbb{E}(|z^n-z^m|^4)=\sum_{i,j,k,l}\mathbb{E}\left( (y^i,y^j)(y^k,y^l)\right).
\end{equation}

\noindent Since $y^l$ is independent of the past, if
for instance $l>\max(i,j,k)$ then $\mathbb{E}((y^i,y^j)(y^k,y^l))=0$.
Therefore, \eqref{k2} reads also

\begin{equation}\label{k3}
\begin{split}
\mathbb{E}(|z^n-z^m|^4)=2\sum_{i,j<k} \mathbb{E}\left( (y^i,y^j)|y^k|^2\right)+4\sum_{i,j<k}\mathbb{E}\left( (y^i,y^k)(y^j,y^k)\right)+\\
4\sum_{i<k} \mathbb{E}\left((y^i,y^k)|y^k|^2\right)+\sum_k \mathbb{E}( |y^k|^4)=D_1+D_2+D_3+D_4.
\end{split}
\end{equation}

\noindent Third step: handling $D_4$ and $D_3$.

The key estimate reads as follows

\begin{equation}\label{keyce}
\mathbb{E}(|y^k|^4|x^k)\leq C N^{-2}.
\end{equation}

\noindent Let us check that \eqref{keyce} is valid. Due to the very properties
of Bernstein polynomials we know that $B_N(1)=1, B_N(X)=x, B_N(X^2)=x^2+\frac{x(1-x)}{N}$
and that $B_N(X^3)=x^3+\frac{3x^2(1-x)}{N}+0(\frac{1}{N^2})$ and
$B_N(X^4)=x^4+\frac{6x^3(1-x)}{N}+0(\frac{1}{N^2})$.
Therefore $B_N((X-x)^4)\leq C N^{-2}$ and since for any function $h$ we have that $\mathbb{E}(h(x^{k+1})|x^k)=h(Ax^k)$
then, due to the very definition of $z^k$

\begin{equation}\label{k4}
\mathbb{E}(|y^k|^4)\leq 4\left( \mathbb{E}(|x^{k+1}-x^k|^4)+\frac{C}{N^4}\right)=O(N^{-2}).
\end{equation}

\noindent Therefore $D_4= O( (n-m)N^{-2})$ and then the result.

For $D_3$ thanks to H\"older inequality, we have the estimate

\begin{equation}\label{k5}
D_3=4\sum_{j<k} \mathbb{E}\left((y^j,y^k)|y^k|^2\right)\leq
 C(n-m) D_4= O( (n-m)^2 N^{-2}).
\end{equation}

\noindent Fourth step: handling $D_1$ and $D_2$.

Using the conditional expectation we have

{
\begin{equation}\label{k6}\begin{split}
D_1=2\sum_{i,j<k} \mathbb{E}\left( (y^i,y^j)|y^k|^2\right)=2\sum_{i,j<k} \mathbb{E}\left( (y^i,y^j)\mathbb{E}\left(|y^k|^2|x^k\right)\right)
\\= 2\sum_{m<k\leq n} \mathbb{E}\left(|z^m-z^k|^2\mathbb{E}(|y^k|^2|x^k)\right).
\end{split}\end{equation}

\noindent Due to \eqref{keyce} and Cauchy-Schwarz inequality
$\mathbb{E}(|y^k|^2 |x^k))=O(N^{-1})$ and it follows

\begin{equation}\label{k7}
D_1\leq C N^{-1} \sum_{k} \mathbb{E}( |z^m-z^k|^2)= C N^{-1} \sum_{k}( \sum_j \mathbb{E}( |y^j|^2)) \leq C N^{-1} \sum_{k} \frac{k-m}{N}\leq  C N^{-2} (n-m)^2.
\end{equation}
}
We now handle $D_2$ exactly as we did for $D_1$. This completes the proof.
\end{proof}

\section*{Acknowledgments}
This work partakes of the research program PEGASE "Percolation et Graphes Al\'eatoires
pour les Syst\`emes Ecologiques" that aims a better understanding
of the role of ecological corridors in biodiversity.
PEGASE is supported by R\'egion Hauts-de-France and FEDER funding.
We also acknowledge the support of CNRS throught MONACAL Prime80's grant.
O.G. is also supported by Labex CEMPI (ANR-11-LABX-
0007-01). The authors thank the Referees for careful reading and useful comments.

\end{document}